\newcommand{\Z}{\mathbb{Z}}
\newcommand{\Q}{\mathbb{Q}}
\newcommand{\legendre}[2]{\left(\frac{#1}{#2}\right)}

\documentclass[11pt]{amsart}
\usepackage{amsmath,amssymb,amsthm,amscd}
\usepackage{tabularx}

\usepackage{thm-restate}

\usepackage{amsmath}
\usepackage{amssymb}
\usepackage{amssymb,amsmath,amsthm,amsfonts,latexsym,euscript}
\usepackage{wasysym}
\usepackage{fancyhdr}
\fancypagestyle{plain}
\usepackage{hyperref}
\newtheorem{theorem}{Theorem}

\newtheorem{lemma}[theorem]{Lemma}
\newtheorem{corollary}[theorem]{Corollary}

\usepackage{amssymb}
\usepackage[T1]{fontenc}
\usepackage[latin1]{inputenc}
\usepackage[usenames,dvipsnames]{xcolor}
\usepackage{graphicx}
\usepackage{setspace}
\usepackage{txfonts}
\usepackage{longtable}
\usepackage{multirow}
\usepackage{blindtext}
\usepackage{ulem}
\usepackage{bm}
\usepackage[margin=1in]{geometry}

\newtheorem*{theorem*}{Theorem}

\AtEndDocument{{\footnotesize%
 \textsc{C. Donnay, Department of Mathematics, Pomona College, Claremont, CA., 91711} \par
  \textit{E-mail address}, C. ~Donnay: \texttt{cvd02014@mymail.pomona.edu} 
  \par
   \textsc{H. Ellers, Department of Mathematics, Harvey Mudd College, Claremont, CA., 91711} \par
  \textit{E-mail address}, H. ~Ellers: \texttt{hellers@g.hmc.edu} 
  \par
   \textsc{K. O'Connor, Department of Mathematics $\&$ Statistics, The College of New Jersey, Ewing, NJ., 08628} \par
  \textit{E-mail address}, K. ~O'Connor: \texttt{oconnk17@tcnj.edu} 
  \par
   \textsc{K. Thompson, Department of Mathematical Sciences, DePaul University, Chicago, IL., 60614} \par
  \textit{E-mail address}, K. ~Thompson: \texttt{kthompson0721@gmail.com} 
  \par
  \textsc{E. Wood, Department of Mathematics, University of Kentucky, Lexington, KY., 40506} \par
    \textit{E-mail address}, E. ~Wood: \texttt{erin.wood@uky.edu} 
 }}

\title{Numbers Represented by a Finite Set of Binary Quadratic Forms}
\author[C. Donnay]{Christopher Donnay}
\author[H. Ellers]{Havi Ellers}
\author[K. O'Connor]{Kate O'Connor}
\author[K. Thompson]{Katherine Thompson}
\author[E. Wood]{Erin Wood}
\begin{document}



\begin{abstract}
Every quadratic form represents 0; therefore, if we take any number of quadratic
forms and ask which integers are simultaneously represented by all members of the
collection, we are guaranteed a nonempty set. But when is that set more than just the
"trivial" {0}? We address this question in the case of integral, positive-
definite, reduced, binary quadratic forms. For forms of the same discriminant, we can
use the structure of the underlying class group. If, however, the forms have different
discriminants, we must apply class field theory.
\end{abstract}
\maketitle



\section{Introduction and Statement of Results}
One of the most fundamental results of number theory and certainly one of the first concerning representation of integers by binary quadratic forms is the following theorem of Fermat:
\begin{theorem*}[Fermat, 1640]
If $n$ is a positive integer, then there are integers $x$ and $y$ so that $x^2+y^2=n$ if and only if in the prime factorization of $n$, $$n= \displaystyle\prod_{i=1}^k p_i^{e_i}$$ all primes $p_i \equiv 3 \pmod{4}$ are raised to even powers. In particular, an odd prime $p$ is of the form $x^2+y^2$ if and only if $p \equiv 1 \pmod{4}$.
\end{theorem*}
The more general question of which integers are represented by a positive definite integral quadratic form is much deeper. 
Restricting our attention to binary forms, one could view Fermat's Theorem in the following way: with finitely-many exceptions, the primes represented by the sum of two squares can be completely determined by congruence conditions. In \cite{GoNI}, conditional on the Generalized Riemann Hypothesis (GRH) and using Geometry of Numbers techniques, the authors determine all $2779$ binary forms with this property and list the specific congruence conditions for a prime to be represented. Moving on to ternary and quaternary forms, there continue to be many famous representation results including Legendre's three-squares theorem \cite{Legendre} and Lagrange's four-squares theorem \cite{Lagrange}. But also, generalizing in a sense these results (respectively) are Dickson's proof of the exact integers represented by the ``Ramanujan ternary'' forms \cite{Dickson}, and the Bhargava-Hanke $290$-Theorem which classifies $6436$ quaternary forms representing all positive integers \cite{BH}. Most recently there is work of Rouse \cite{Rouse} which classifies quadratic forms representing all odd positive integers, and work of Barowsky et al \cite{TAU} which lists $73$ explicit pairs of integers which can be excepted by a positive definite quadratic form with even coefficients.\\
\\
\noindent Each of these results, however, either begins with a fixed form and classifies integers represented or begins with a collection of integers to be represented and constructs desired forms. There is very little in the literature, however, to be found on determining the integers simultaneously represented by more than one form. Certainly, for instance, if one restricts to primes and to forms listed in \cite{GoNI} this can be made explicit. There are also situations where via completely elementary methods one can show that the only integer $m$ simultaneously represented by two binary forms is $m=0$; this is the case, for instance, in considering the forms $x^2+5y^2$ and $2x^2+2xy+3y^2$. The purpose of this paper is to explore cases beyond these; that is, given a finite collection of binary integral positive definite quadratic forms, when is a nonzero integer simultaneously represented by all forms in the collection?\\
\\
\noindent We began by considering the integers simultaneously represented by positive definite binary integral quadratic forms (which for the rest of this section we simply call ``forms'') of the same discriminant $\Delta$. This leads to our first result, which gives a test to see when the intersection of represented integers is $0$:

\begin{theorem}
\label{Test}
Let $S_{\Delta}$ be the set of all forms of discriminant $\Delta<0$ with class number $h(\Delta) \geq 2$. For $i \not = j$ define $Q_{ij}(x,y,z,w):= Q_i(x,y) -Q_j(z,w)$ for $Q_i, Q_j \in S_{\Delta}$. The only non-negative integer $m$ represented by all $Q \in S_{\Delta}$ is $m=0$ if and only if at least one of the $Q_{ij}$ is anisotropic.
\end{theorem}
We then found that if we have a fundamental discriminant, we can tell exactly when the only integer represented by the forms of a given discriminant is $0$.

\begin{theorem}\label{FDT}[Fundamental Discriminant Theorem]
Let $\Delta<0 \equiv 0, 1 \pmod{4}$ be a fundamental discriminant, and let $S_\Delta$ be the set of all reduced forms of discriminant $\Delta$. If the class number $h(\Delta)$ is even, then $m=0$ is the only integer represented by all forms in $S_\Delta$. 
\end{theorem}
Of course, this theorem only concerns the case that $h(\Delta)$ is even. Towards understanding the situation when $h(\Delta)$ is odd, we had the following results confirming that if an integer $n$ were represented, then infinitely-many prime multiples of $n$ would also be represented. 

\begin{theorem}
\label{POnetoAll}
Let $p$ be a prime and let $S_{\Delta}$ be the set of forms of discriminant $\Delta$. If $n$ is represented by all forms $Q \in S_{\Delta}$ and $p$ is represented by some form $Q_i \in S_{\Delta}$, then $np$ is represented by all forms $Q \in S_{\Delta}$.
\end{theorem}

Using this last theorem we were able to address the situation of integers represented by all forms in $S_{\Delta}$, where $h(\Delta)$ is odd: 
\begin{theorem}{\label{OddClSize2}}
Let $S_{\Delta}$ be the collection of all forms of discriminant $\Delta$, and suppose $h(\Delta)$ is odd. Then there are infinitely-many positive integers represented by all forms in $S_\Delta$.
\end{theorem}
Note: When this statement is proved, we will be more explicit about a nonzero positive integer represented by all forms. Also, Theorem \ref{OddClSize2} has a corollary crucial for the introductory consideration of forms of different discriminants:
\begin{corollary}{\label{OddClSizeCor}}
Let $Q_1$ and $Q_2$ be forms of discriminants $\Delta_1$ and $\Delta_2$ (respectively), and suppose that $h(\Delta_1)$ and $h(\Delta_2)$ are both odd. Then there are infinitely many $m$ represented by both $Q_1$ and $Q_2$.
\end{corollary}
Towards additional results when $\Delta_1 \neq \Delta_2$, we then looked at the relationship between the local representation of integers and the integral representation of multiples of these integers. By Hasse-Minkowski a local representation implies a $\mathbb Q$ representation; by clearing denominators as necessary this guarantees that a nonzero, \textit{square} multiple of any locally represented integer will be integrally represented. But towards guaranteeing a nonsquare multiple will be represented we have:
\begin{theorem} {\label{LocalOneD}}
Let $Q$ be a quadratic form, and suppose $m$ is locally represented by $Q$. Then some nonzero, nonsquare multiple of $m$ is represented by $Q$.
\end{theorem}
\begin{theorem} {\label{LocalTwoD}}
Let $Q_1$ and $Q_2$ be forms with discriminants $\Delta_1$ and $\Delta_2$, respectively. Let $m$ be locally represented by $Q_1$ and $Q_2$, and let $n \not \equiv 0 \pmod{m}$ be represented by all forms of discriminant $\Delta_1$ and $\Delta_2$. Then $mk$ is represented by $Q_1$ and $Q_2$, where $k$ is a non-square integer.
\end{theorem}

Last, to consider forms of different discriminants and with at least one even class group, we used a relationship between representation by a quadratic form and splitting behavior in the Hilbert class field: 

\begin{theorem}
\label{number of ideals}
Let $\Delta=s^2n$, where $n<0$ is square-free, and let $Q$ be some form of discriminant $\Delta$ and order $m$. Let $K=\Q(\sqrt{n})$ and $L$ be the Hilbert class field of $K$. A prime $p$ not dividing the discriminant of $K$ is represented by $Q$ if and only if $p$ splits into $\frac{[L:\Q]}{m}$ factors in $L$.
\end{theorem}

This paper is organized as follows: Section \ref{Background} provides general background on quadratic forms needed for the proofs of Theorems \ref{Test} through \ref{LocalTwoD}. The next section will then provide the proofs of those theorems. Section \ref{CFT} provides additional background on class field theory needed for the following section's proof of Theorem \ref{number of ideals}. Last, Section \ref{Examples} will include examples.\\

\textbf{Acknowledgements:} This research was supported by the National Science Foundation (DMS-1461189). We additionally thank Jeremy Rouse for his helpful discussions and suggestions.
\section{Background}
\label{Background}
An \textbf{integral binary quadratic form} is a homogeneous polynomial
\begin{eqnarray*}
Q: \mathbb Z^2 & \to & \mathbb Z \\
(x,y) & \mapsto& ax^2+bxy+cy^2
\end{eqnarray*}
where $a,b,c \in \mathbb Z$. An integer $m$ is \textbf{represented} by a form $Q$ if there exist integers $x$ and $y$ such that $Q(x,y)=m$, an we say $m$ is \textbf{properly represented} by $Q$ if $x$ and $y$ are relatively prime. Last, an integer $m$ is \textbf{locally represented} by a form $Q$ if $m$ is represented over each $p$-adic ring $\mathbb{Z}_p$.

Using the notation of \cite{Lam}, we let $D(Q) = \{ m \in \mathbb Z-\{0\} : \exists (x,y) \in \mathbb Z^2, Q(x,y)=m \}$. Then if we let $S = \{Q_1,\ldots, Q_n\}$ be a finite set of binary quadratic forms, we will define the \textbf{intersection of $S$} $\textrm{int}(S)$ as $$\textrm{int}(S) = \{ m \in \mathbb Z : m \in D(Q) \textrm{ for all } Q \in S \}.$$
We say $\textrm{int}(S)$ is \textbf{trivial} when $\textrm{int}(S) = \emptyset$. Equivalently, an intersection $\textrm{int}(S)$ is trivial if and only if the only integer $m$ represented by all $Q \in S$ is $m=0$. 

We say that $Q$ is \textbf{positive definite} if $Q(x,y)>0$ for all $(x,y) \neq (0,0)$. Moreover, $Q$ is \textbf{primitive} if $\gcd(a,b,c)=1$, and is \textbf{reduced} if the following two conditions hold:
\begin{itemize}
\item[(R1)] $\vert b \vert \leq a \leq c$,
\item[(R2)] $b \geq 0$ if $\vert b \vert =a$ or if $a=c$.
\end{itemize}
Henceforth, by ``form'' we mean ``integral, binary, positive definite, {reduced}, primitive quadratic form.'' \\

\noindent The \textbf{discriminant} of $Q$ is defined to be $\Delta := b^2-4ac$.  Note that $\Delta \equiv 0,1 \pmod{4}$ always, and when $Q$ is positive definite, $\Delta < 0$. A discriminant $\Delta$ is said to be \textbf{fundamental} if and only if either
\begin{itemize}
\item[(i)] $\Delta \equiv 1 \pmod{4}$ is square free, or
\item[(ii)] $\Delta =4n$, with $n \equiv 2, 3 \pmod{4}$ square free.
\end{itemize}
One easily can see that for fixed discriminant $\Delta$ (fundamental or otherwise) there are only finitely many forms $Q$ of that discriminant. We therefore denote by $S_\Delta$ this finite collection of forms. \\


We can restrict our attention to ``forms'' for the following reason. There is an equivalence relation between all positive definite quadratic forms of the same discriminant, given by a $\mathbb Z$-linear change of variable. If that change-of-base matrix can be written with determinant $1$ the equivalence is said to be \textbf{proper}; otherwise, it is said to be \textbf{improper}. It is clear that equivalent forms represent exactly the same integers. But moreover, if we restrict ourselves to \textbf{proper equivalence classes} each class will have exactly one reduced form \cite[Thm 2.8]{Cox}.\\ 

There is one well-documented fact regarding the representation of an integer $m$ by a form $Q$ and the discriminant $\Delta$ of $Q$. Since we will refer to this result often, we state it now (with wording taken from \cite[Lemma 2.5]{Cox}):
\begin{lemma}
\label{Cox Lemma 2.5}
Let $\Delta \equiv 0,1 \pmod{4}$ be an integer and $m$ be an odd integer relatively prime to $\Delta$. Then $m$ is properly represented by a primitive form of discriminant $\Delta$ if and only if $\Delta$ is a quadratic residue modulo $m$.  
\end{lemma}
\noindent
For fixed discriminant $\Delta<0$, the finite collection of proper equivalence classes of forms (with representatives the reduced forms) form an abelian group, called the \textbf{class group of $\Delta$}. We denote the class group as $C(\Delta)$, its size as $h(\Delta)$ and with slight abuse of notation we use $Q(x,y)$ to mean $[Q(x,y)] \in C(\Delta)$. The operation on this group is composition of forms--first defined by Gauss in \cite{Gauss} and which was much later generalized by Bhargava in \cite{Manjul}. For $Q_1(x_1,y_1)=a_1x^2+b_1xy+c_1y^2$, and $Q_2(x_2,y_2)=a_2x^2+b_2xy+c_2y^2 \in C(\Delta)$ we define 
\begin{eqnarray*}
Q_1(x_1,y_1) \circ Q_2(x_2,y_2) & = & F(B_1(x_1,y_1;x_2,y_2),B_2(x_1,y_1;x_2,y_2))
\end{eqnarray*}
where $B_i(x_1,y_1;x_2,y_2)=a_ix_1x_2+b_ix_1y_2+c_iy_1x_1+d_iy_1y_2$, for $i=1,2$. 


The identity element of $C(\Delta)$ is the class containing the $\textbf{principal form}$, defined by
\begin{eqnarray*}
x^2-\frac{\Delta}{4}y^2, & \Delta \equiv 0 \pmod{4},\\
x^2+xy+\frac{1-\Delta}{4}y^2, & \Delta \equiv 1 \pmod{4}.
\end{eqnarray*}
Given a form $Q(x,y)=ax^2+bxy+cy^2$, the \textbf{inverse} of $Q$ is $Q^{-1}(x,y)=ax^2-bxy+cy^2$. $Q$ and $Q^{-1}$ are always improperly equivalent, and are properly equivalent if and only if $Q$ has order $2$ in $C(\Delta)$.\\

\section{Proofs of Theorems \ref{Test}-\ref{LocalTwoD} and Corollary \ref{OddClSizeCor}}
Before the proof of Theorem \ref{Test} we recall its statement:\\ 

\textbf{Theorem \ref{Test}.} \textit{Let $S_{\Delta}$ be the set of all forms of discriminant $\Delta<0$ with class number $h(\Delta) \geq 2$. For $i \not = j$ define $Q_{ij}(x,y,z,w):= Q_i(x,y) -Q_j(z,w)$ for $Q_i, Q_j \in S_{\Delta}$. The only non-negative integer $m$ represented by all $Q \in S_{\Delta}$ is $m=0$ if and only if at least one of the $Q_{ij}$ is anisotropic.}\\
\\
Note: A quadratic form $Q$ is \textbf{anisotropic} if for all $\vec{x} \in \mathbb{Z}^n$ with $\vec{x} \not = \vec{0}$, $Q(\vec{x}) \not = 0$. If a form is not anisotropic, it is \textbf{isotropic}. 

\begin{proof}
Let $n = h(\Delta)$. First suppose for all $i,j$ with $1 \leq i,j \leq n$, $Q_{ij}$ is isotropic. If $n = 2$, there is only one form $Q_{1,2}$, so if that form is isotropic then the intersection of the forms is non-trivial. So now suppose $n > 2$.  Without loss of generality, let $Q_1$ be the principal form. Then $Q_1$ represents $(n-1)$ nonzero (not necessarily distinct) values $a_2, \ldots , a_n$, where $Q_1$ and $Q_k$ both represent $a_k$ for all $2 \leq k \leq n$. By composing $Q_1$ with itself $(n-2)$ times, we have that $Q_1$ represents
\[\displaystyle\prod_{\substack{m=2 \\ m\not = k}}^{n} a_m.\]
Composing once more with $Q_k$ we have that $Q_k$ represents
\[\displaystyle\prod_{m=2}^{n} a_m.\]
Thus, all forms of discriminant $\Delta$ represent the nonzero integer $\prod_{m=2}^{n} a_m$. 
Conversely, suppose there exists an $i,j$ such that $Q_{ij}$ is anisotropic. Then $\textrm{int}(\{Q_i,Q_j\}) = \emptyset$, and so $\textrm{int}(S_{\Delta}) = \emptyset$.
\end{proof}
Before proceeding with the proof of Theorem \ref{FDT}, we introduce the following lemma:
\begin{lemma}
\label{OnetoAll}
Let $Q(x,y) = qx^2+bxy+cy^2$ be a form of discriminant $\Delta=-2^k p_1 \dots p_{\ell} $, where $p_i$ for $1 \leq i \leq \ell$ and $q$ are odd primes with $q \nmid \Delta$. 
If $\left( \frac{q}{p_i} \right)=1$ for $1 \leq i \leq \ell$ and $n$ is represented by $Q$, then either $\left( \frac{n}{p_i} \right)=1$ or $n \equiv 0 \pmod{p_i}$. Similarly, if $\left( \frac{q}{p_i} \right)=-1$ and $n$ is represented by $Q$, then either $\left( \frac{n}{p_i} \right)=-1$ or $n \equiv 0 \pmod{p_i}$. Therefore, if $q$ is a quadratic residue modulo $\Delta$ then $n$ will be a quadratic residue as well, and if $q$ is a quadratic non-residue modulo $\Delta$ then $n$ will also be a quadratic non-residue.
\end{lemma}

\begin{proof}
(of Lemma \ref{OnetoAll})
Consider the binary quadratic form $Q(x,y) = qx^2+bxy+cy^2$ of discriminant $\Delta=-2^k p_1 \dots p_\ell$ where $p_i$ for $1 \leq i \leq \ell$ and $q$ are prime, with $q \nmid \Delta$. We have
\begin{eqnarray*}
b^2-4qc & = & \Delta \\
c & = & \frac{-\Delta +b^2}{4q}.
\end{eqnarray*}
Substituting this for $c$ in $Q_1$ and completing the square, we have
\[q\left(x+\frac{b}{2q}y\right)^2+\frac{-\Delta }{4q}y^2.\]
Now suppose $n=q(x+\frac{b}{2q}y)^2+\frac{-\Delta }{4q}$ so that $n$ is represented by $Q$. Since $q$ and $p_i$ are relatively prime, we have
$n \equiv q(x+\frac{b}{2q}y)^2 \pmod{p_i}$ for all $1 \leq i \leq \ell$.
The Legendre symbol is
\[\left(\frac{n}{p_i}\right)=\left(\frac{q(x+\frac{b}{2q}y)^2}{p_i}\right).\]
Thus, $\legendre{n}{p_i} = \legendre{q}{p_i}$ or $n|\Delta$.
\end{proof}

Now we recall the statement of Theorem \ref{FDT} and provide its proof:\\

\textbf{Theorem \ref{FDT}.} \textit{Let $\Delta<0 \equiv 0, 1 \pmod{4}$ be a fundamental discriminant, and let $S_\Delta$ be the set of all reduced forms of discriminant $\Delta$. If the class number $h(\Delta)$ is even, then $m=0$ is the only integer represented by all forms in $S_\Delta$.}

\begin{proof}
For the following proof, we will define a quadratic residue as follows: $a$ is a quadratic residue modulo $m$ if $a \not \equiv 0 \pmod{m}$ and if there exists a solution to $x^2 \equiv a \pmod{m}$. Similarly, $a$ is a quadratic non-residue modulo $m$ if $a \not \equiv 0 \pmod{m}$ and there does not exist a solution to $x^2 \equiv a \pmod{m}$. \\

Let $\Delta$ be a fundamental discriminant, and let $S_{\Delta}$ be the set of forms of discriminant $\Delta$, where $h(\Delta)$ is even. In order for $h(\Delta)$ to be even, one of the following cases must hold:
\begin{itemize}
\item[\textbf{Case 1:}] $\Delta=-p_1 \dots p_n \equiv 1 \pmod{4}$ for $n \geq 2$,
\item[\textbf{Case 2:}] $\Delta=-8s$ where $s=p_1 \dots p_n$, such that $p_i$ is odd for all $i$, and $n \geq 1$,
\item[\textbf{Case 3:}] $\Delta=-4s$ where $s=p_1 \dots p_n$, such that $s \equiv 1 \pmod{4}$ and $n \geq 1$, or $s \equiv 3 \pmod{4}$ and $n \geq 2$.
\end{itemize}

\textbf{Case 1:} The principal form here is $Q_1(x,y)=x^2+xy+\frac{1+p_1 \dots p_n}{4}y^2$. Completing the square and examining $Q_1 \pmod{p_i}$ for all $1 \leq i \leq n$, we find $$Q_1(x,y) \equiv \left(x+\frac{1}{2}y\right)^2 \pmod{p_i}.$$ Hence $Q_1$ only represents quadratic residues modulo $p_i$ or values congruent to $0$ modulo $p_i$ which in turn implies $Q_1$ represents quadratic residues or values congruent to $0 \pmod{p_1 \dots p_n}$. \\

We are also guaranteed a form that represents quadratic non-residues modulo $p_i$ for all $i$. Let $q$ be an odd prime such that $\left(\frac{q}{p_i}\right)=-1$ for all $1 \leq i \leq n$ and $\left( \frac{-p_1 \dots p_n}{q} \right)=1$. Choose $b,c \in \mathbb{Z}$ such that $b^2-4qc=-p_1 \dots p_n$. Then the form $Q_2(x,y)=qx^2+bxy+cy^2$ represents $q$, a quadratic non-residue modulo $p_i$. Then by Lemma \ref{OnetoAll}, $Q_2$ only represents quadratic non-residues and values congruent to $0\pmod{p_i}$ which again implies $Q_2$ represents quadratic non-residues and values congruent to $0 \pmod{p_1 \dots p_n}$. So if $m \in \textrm{int}(\{Q_1, Q_2 \})$ then $m \equiv 0 \pmod{p_1 \dots p_n}$. \\

We now claim that if $m$ is represented by $Q_1$ and $m\equiv 0 \pmod{p_1 \dots p_n}$ then $\frac{m}{p_1 \dots p_n}$ is represented by $Q_1$. Considering such an $m$ we have:
\begin{eqnarray*}
m&=&x^2+xy+\frac{1+p_1 \dots p_n}{4}y^2\\
0&\equiv & \left(x+\frac{1}{2}y \right)^2 \pmod{p_1 \dots p_n}\\
-2x & \equiv & y \pmod{p_1 \dots p_n}.
\end{eqnarray*}
So either $x,y \not\equiv 0 \pmod{p_1 \dots p_n}$ or $x\equiv y \equiv 0 \pmod{p_1 \dots p_n}$. If $x,y \not \equiv 0 \pmod{p_1 \dots p_n}$, then $y=-2x+k(p_1 \dots p_n)$ for some $k \in \mathbb Z$ which gives
\begin{eqnarray*}
m&=&x^2+x(-2x+k(p_1 \dots p_n))+ \frac{1+p_1 \dots p_n}{4}(-2x+k(p_1 \dots p_n))^2\\
& \equiv &(p_1 \dots p_n)x^2 \pmod{(p_1 \dots p_n)^2}.
\end{eqnarray*} 
Recall that $m \equiv 0 \pmod{p_1 \dots p_n}$. So either $m=(p_1 \dots p_n)\ell$ or $m=(p_1 \dots p_n)^r \ell$ for $(p_1 \dots p_n) \nmid \ell$ and $r>1 \in \mathbb{Z}$. We first handle the case that $r>1$. Then 
\begin{eqnarray*}
(p_1 \dots p_n)^r \ell &\equiv& (p_1 \dots p_n)x^2 \pmod{(p_1 \dots p_n)^2}\\
x^2 &\equiv& 0 \pmod{p_1 \dots p_n}.
\end{eqnarray*}
This is a contradiction since we supposed $x,y \not \equiv 0 \pmod{p_1 \dots p_n}$. So we must instead have $x \equiv y \equiv 0 \pmod{p_1 \dots p_n}$. This means $x=i(p_1 \dots p_n)$ and $y=j(p_1 \dots p_n)$ for some $i,j \in \mathbb{Z}$. Hence \begin{eqnarray*}
\left(p_1\ldots p_n i+\frac{1}{2}p_1\ldots p_n j\right)^2+\frac{p_1\ldots p_n}{4}(p_1\ldots p_n j)^2 &=& m\\
\left(i+\frac{1}{2}j\right)^2+\frac{p_1\ldots p_n}{4}j^2 & = & \frac{m}{(p_1\ldots p_n)^2}.
\end{eqnarray*}
Therefore, if $m=(p_1 \dots p_n)^r \ell$ is represented by $Q_1$ then $\frac{m}{(p_1 \dots p_n)^2}$ is represented by $Q_1$. But since $Q_1$ represents $(p_1 \dots p_n)$ then by the composition law, $Q_1$ represents $m(p_1 \dots p_n)$. Thus, if $m$ is represented by $Q_1$ and $m=(p_1 \dots p_n)^r \ell $ then $\frac{m}{(p_1 \dots p_n)}$ is represented by $Q_1$. 

Before we address the case where $m=(p_1 \dots p_n)\ell$ with $p_1 \dots p_n \nmid \ell$, we show that if $m$ is represented by $Q_2$ and $m=(p_1 \dots p_n)^r \ell$ for $r>1$, then $\frac{m}{(p_1 \dots p_n)}$ is also represented by $Q_2$. By a similar argument as above, we know that if $Q_2(x,y)=m$ for some $x,y \in \mathbb{Z}$ then $x,y \equiv 0 \pmod{p_1 \dots p_n}$. Again, let $x=i(p_1 \dots p_n)$ and $y=j(p_1 \dots p_n)$ for some $i,j \in \mathbb{Z}$ and substitute to get: 
\begin{eqnarray*}
q\left(p_1\ldots p_n i+\frac{1}{2}p_1\ldots p_n j\right)^2+\frac{p_1\ldots p_n}{4}(p_1\ldots p_n j)^2 &=& m\\
q\left(i+\frac{1}{2}j\right)^2+\frac{p_1\ldots p_n}{4}j^2 & = & \frac{m}{(p_1\ldots p_n)^2},
\end{eqnarray*}

so $\frac{m}{(p_1 \dots p_n)^2}$ is represented by $Q_2$. By the composition law, since $p_1 \dots p_n$ is represented by $Q_1$ (and since $Q_1$ is the principal form), $p_1 \dots p_n m$ is represented by $Q_2$. Thus, if $m$ is represented by $Q_2$ and $m=(p_1 \dots p_n)^r \ell$ for $r>1$, then $\frac{m}{p_1 \dots p_n} $ is represented by $Q_2$. 

Now let $m$ be represented by both $Q_1$ and $Q_2$, and suppose that $m=(p_1 \dots p_n)^r \ell$, where $r>1$ and $\ell \not \equiv 0 \pmod{p_1 \dots p_n}$ or $\ell=0$. We can divide out all of the factors of $(p_1 \dots p_n)$ and $\ell$ will still be represented by both $Q_1$ and $Q_2$. This implies that for all $p_i$, $\ell$ is both a quadratic residue and a quadratic non-residue modulo $p_i$ or $\ell \equiv 0 \pmod{p_i}$. So $\ell \equiv 0 \pmod{p_1 \dots p_n}$, which implies $\ell=0$. Hence, $m=(p_1 \dots p_n)^r \ell$ is only represented by $Q_1$ and $Q_2$ when $\ell=0$. 

As for the remaining case that $m=p_1 \dots p_n \ell$ where $\ell \not \equiv 0 \pmod{p_1 \dots p_n}$ or $\ell=0$:  if $m$ is represented by both $Q_1$ and $Q_2$ then again by composition $(p_1 \dots p_n)^2m$ is represented by both forms. This will then go back to an earlier argument, and we conclude that $m=0$. 

\textbf{Case 2:} Let $\Delta=-8s$ where $s=p_1 \dots p_n$, such that $p_i \neq 2$ for all $i$ and $n \geq 1$. The principal form is $Q_3(x,y)=x^2+2sy^2$. We see immediately that $Q_3(x,y) \equiv x^2 \pmod{p_i}$ for all $1 \leq i \leq n$. Hence $Q_3$ 
represents quadratic residues and values congruent to $0$ modulo $s$. We also have the form $Q_4(x,y)=qx^2+bxy+cy^2$, where $\left(\frac{q}{p_i}\right)=-1$ for all $i$ and $b \not \equiv 0 \pmod{s}$. Since $Q_4$ represents $q$, a quadratic non-residue modulo $p_i$ for all $i$ then $Q_4$ represents quadratic non-residues or values congruent to $0$ modulo $s$. Therefore the only integers represented by both $Q_3$ and $Q_4$ are congruent to $0 \pmod{s}$. We will for this argument consider a third form $Q_3'(x,y) = 2x^2+sy^2$ which is distinct from the principal form but which still has determinant $\Delta$. 

We claim now that if $m \equiv 0 \pmod{s}$ is represented by $Q_3$  and $Q_3'$ then $\frac{m}{s}$ is represented by $Q_3$ and $Q_3'$. Considering $Q_3$ we have 
\begin{eqnarray*}
m & = & x^2+2sy^2\\
0 & \equiv & x \pmod{s}.
\end{eqnarray*}
This implies $x=si$ for some $i \in \mathbb{Z}$. So
\begin{eqnarray*}
m & = & (s)^2i^2+2sy^2\\
\frac{m}{s} & =& si^2+2y^2.
\end{eqnarray*}
Hence $\frac{m}{s}$ is represented by the form $Q_3'$. 
Now let $n \equiv 0 \pmod{s}$ be represented by $Q_3'$. Then 
\begin{eqnarray*}
n & = & 2x^2+sy^2\\
0 & \equiv & x \pmod{s}.
\end{eqnarray*}
This implies $x=si$ for some $i \in \mathbb{Z}$. So
\begin{eqnarray*}
n & = & 2(s)^2i^2+sy^2\\
\frac{n}{s} & =& 2si^2+y^2.
\end{eqnarray*}
Hence $\frac{n}{s}$ is represented by the form $Q_3$. More crucially, $\frac{m}{s^2}$ is represented by $Q_3$ if $m$ is represented by $Q_3$ and if $m=s^{2k+1} \ell$ is represented by both $Q_3$ and $Q_3'$ then $\frac{m}{s^{2k+1}}$ is represented by both forms.

Now suppose $m$ is represented by $Q_4$ and $m \equiv 0 \pmod{s}$. Either $m=s^rl$ for $r>1$ and $l \not \equiv 0 \pmod{s}$, or $m=sl$ for $l \not \equiv 0 \pmod{s}$. Assume the former. Then 
\begin{eqnarray*}
m&=&q\left(x+\frac{b}{2q}y\right)^2+\frac{2s}{q}y^2\\ 
0 &\equiv& \left(x+\frac{b}{2q}y\right)^2 \pmod{s}. 
\end{eqnarray*}
This implies that 
$y \equiv \frac{-2q}{b}x \pmod{s}$ and so either $x\equiv y \equiv 0 \pmod{s}$ or $x,y \not \equiv 0 \pmod{s}$. First suppose $x,y \not \equiv 0 \pmod{s}$. Thus $y=\frac{-2q}{b}x+sf$ for some integer $f$. Substituting into $Q_4$ we get
\begin{eqnarray*}
m & = & qx^2 +bx\left(\frac{-2q}{b}x+sf\right)+c\left(\frac{-2q}{b}x+sf\right)^2\\
0 & \equiv & \frac{-8qs}{b^2}x^2 \pmod{s^2}\\
0 & \equiv & x^2 \pmod{s^2}\\
0 & \equiv & x \pmod{s}.
\end{eqnarray*}
This is a contradiction since we supposed $x,y \not \equiv 0 \pmod{s}$. Hence $x\equiv y \equiv 0 \pmod{s}$. This implies $x=si$ and $y=sj$ for some $i,j \in \mathbb{Z}$ and 
\begin{eqnarray*}
q\left(si + \frac{b}{2q}sj\right)^2 + \frac{2s}{q}(sj)^2 & = & m\\
q\left(i+\frac{b}{2q}j\right)^2+\frac{2s}{q}j^2 & = & \frac{m}{s^2}
\end{eqnarray*}
Therefore, if $m$ is represented by $Q_4$ and $m=s^r \ell$ then $\frac{m}{s^2}$ is represented by $Q_4$. Since $s$ is represented by the principal form, by composition, if $Q_4$ represents $\frac{m}{s^2}$ then $Q_4$ represents $\frac{m}{s}$. 

If $Q_3'$ represents quadratic non-residues and values congruent to $0\pmod{s}$: Let $m=s^k\ell$ such that $\ell \not \equiv 0 \pmod{s}$ or $\ell=0$ be represented by both $Q_3$ and $Q_3'$. Then we can divide off all of the factors of $s$, and $\ell$ will be represented by both forms. This implies that $\ell$ is both a quadratic residue and quadratic non-residue modulo $s$, or $0$ modulo $s$. So $\ell=0$ and $m=s^r \ell$ is only represented by $Q_3$ and $Q_3'$ when $\ell=0$.

If $Q_3'$ represents quadratic residues and values congruent to $0$ modulo $s$: Let $m=s^k\ell$, $k>1$ such that $\ell \not \equiv 0 \pmod{s}$ or $\ell=0$ be represented by  $Q_3$, $Q_3'$ and $Q_4$. We can divide out all factors of $s$ and $\ell$ will still be represented by $Q_4$ and either $Q_3$ or $Q_3'$. This implies that $l$ is both a quadratic residue and quadratic non-residue modulo $s$, or $0$ modulo $s$. So $\ell=0$ and $m=(s)^r \ell$ is only represented by $Q_3$, $Q_3'$ and $Q_4$ when $\ell=0$.

As for the remaining situation where $m=s\ell$ where $\ell \not \equiv 0 \pmod{s}$ or $\ell =0$, if such an $m$ were represented by $Q_3$, $Q_3'$ and $Q_4$ then by the composition law $s^2m$ is represented by all forms. This then implies $\ell = m =0$.

\textbf{Case 3:} Let $\Delta =-4s$ where $s=p_1 \dots p_n$ such that $s \equiv 1 \pmod{4}$ and $n \geq 1$. Now the principal form is $Q_5(x,y)=x^2+sy^2$ which once again 
represents quadratic residues, and values congruent to $0$ modulo $s$. There is also the form $Q_6(x,y)=qx^2+bxy+cy^2$ where $\left(\frac{q}{p_i}\right)=-1$ for all $i$ and $b \not \equiv 0 \pmod{s}$. Again, $Q_6$ only represents quadratic non-residues or values congruent to $0$ modulo $s$. So the only integers represented by both $Q_5$ and $Q_6$ are congruent to $0\pmod{s}$.

We claim that if $m$ is represented by $Q_5$ and $m \equiv 0 \pmod{s}$, then $\frac{m}{s}$ is represented by $Q_5$ because 
\begin{eqnarray*}
m & = & x^2 + sy^2\\
0 & \equiv & x \pmod{s}.
\end{eqnarray*}
Hence $x=is$ for some $i \in \mathbb{Z}$ and 
\begin{eqnarray*}
m & = & (si)^2 + sy^2\\
\frac{m}{s} & = & si^2+y^2.
\end{eqnarray*}

Similarly, we claim that if $m$ is represented by $Q_6$ and $m \equiv 0 \pmod{s}$, then $\frac{m}{s}$ is represented by $Q_6$. Either $m=s^k \ell$ for $k>1$, $\ell \not \equiv 0 \pmod{s}$, or $m=s\ell$ where $\ell \not \equiv 0 \pmod{s}$. Assume the former. Then \begin{eqnarray*}
m & = & qx^2+bxy+cy^2\\
0 & \equiv & q\left(x+\frac{b}{2q}y\right)^2 \pmod{s}\\
\frac{-2q}{b}x & \equiv & y \pmod{s}.
\end{eqnarray*}

Assume $x,y \not \equiv 0 \pmod{s}$. Then $y=\frac{-2q}{b}x+sf$ for some $f \in \mathbb{Z}$ and 
\begin{eqnarray*}
m & = & qx^2+bx\left(\frac{-2q}{b}x+sf\right)+c\left(\frac{-2q}{b}x+sf\right)^2\\
0 & \equiv & \frac{4qs}{b^2}x^2 \pmod{s^2}\\
x & \equiv & 0 \pmod{s}.
\end{eqnarray*}

This is a contradiction, so it must be that $x,y \equiv 0 \pmod{s}$. Substituting $x=is$ and $y=js$ for some $i,j \in \mathbb{Z}$ gives
\begin{eqnarray*}
m & = & q\left(si+\frac{bs}{2q}j\right)^2+\frac{s}{q}(sj)^2\\
\frac{m}{s^2} & = & q\left(i+\frac{b}{2q}j\right)^2+\frac{s}{q}j^2.
\end{eqnarray*}

Therefore if $m$ is represented by $Q_6$ and $m=s^k \ell$ for $k>1$ and $\ell \not \equiv 0 \pmod{s}$, then $\frac{m}{s^2}$ is also represented by $Q_6$. Since $s$ is represented by the principal form, by composition, if $Q_6$ represents $\frac{m}{s^2}$, then $Q_6$ represents $\frac{m}{s}$. 

Let $m$ be represented by both $Q_5$ and $Q_6$, and suppose $m=s^k \ell$ for $k>1$, where $\ell \not \equiv 0 \pmod{s}$ or $\ell=0$. We can divide out all factors of $s$, and $\ell$ will still be represented by both forms. This implies that $\ell$ is both a quadratic residue and a quadratic non-residue modulo $s$, or $0$ modulo $s$. This means that $\ell \equiv 0 \pmod{s}$, and hence $\ell=0$. Therefore, $m=s^k \ell$ is only represented by $Q_5$ and $Q_6$ when $\ell=0$. 

We now address the case when $m=s\ell$ where $\ell \not \equiv 0 \pmod{s}$ or $\ell=0$. Suppose $m$ is represented by both $Q_5$ and $Q_6$. Then by composition, $s^2m$ is represented by both forms. This can only be the case when $\ell=0$ by the above argument. Hence if $m$ is represented by both forms, then $m=0$. 

\end{proof} 


Moving on to cases where the intersection is nontrivial, we first have:\\

\textbf{Theorem \ref{POnetoAll}.} \textit{Let $p$ be a prime and let $S_{\Delta}$ be the set of forms of discriminant $\Delta$. If $n$ is represented by all forms $Q \in S_{\Delta}$ and $p$ is represented by some form $Q_i \in S_{\Delta}$, then $np$ is represented by all forms $Q \in S_{\Delta}$.}

\begin{proof} 
Let $Q_1 \in S_{\Delta}$ and suppose that the prime $p$ is represented by $Q_1$. For any $Q_2 \in S_{\Delta}$ (not necessarily distinct from $Q_1$) there exists a $Q_3 \in S_{\Delta}$ 
such that $Q_1 \circ Q_3 =Q_2$. Since $n$ is represented by $Q_3$ and $p$ is represented by $Q_1$, $np$ is represented by $Q_2$. Thus, $np$ is represented by all forms with discriminant $\Delta$. Note that there are an infinite number of primes represented by some form in $S_{\Delta}$ by Lemma \ref{Cox Lemma 2.5} and Dirichlet's theorem on primes in arithmetic progressions.
\end{proof}


%
%

Before giving the proof of Theorem \ref{OddClSize2}, we have the following lemma:
\begin{restatable}{lemma}{UniqueSq}
\label{UniqueSq}
If $\Delta$ is such that $h(\Delta)$ is odd, then for all forms $Q_j$ of discriminant $\Delta$ there exists some form $Q_i$ of discriminant $\Delta$ such that ${Q_i}^2=Q_j$.
\end{restatable}

\begin{proof}
Since the class group has odd order, no element can be its own inverse. Also, it follows from elementary group theory that the square of each class is distinct. In particular, if $i \neq k$ then ${Q_i}^2 \neq {Q_k}^2$. Since there are $h(\Delta)$ proper equivalence classes, and $h(\Delta)$ possible values for ${Q_i}^2$, for all $Q_j \in C(\Delta)$ there must exist some form $Q_i$ such that ${Q_i}^2=Q_j$.
\end{proof}


\textbf{Theorem \ref{OddClSize2}.} \textit{Let $S_{\Delta}$ be the collection of all forms of discriminant $\Delta$, and suppose $h(\Delta)$ is odd. Then there are infinitely-many positive integers represented by all forms in $S_\Delta$. In particular, the product of the $x^2$ and $y^2$ coefficients of all the non-principal, non-equivalent forms of discriminant $\Delta$ will be represented.}

\begin{proof}
Let the elements of $S_\Delta$ be: 
\begin{eqnarray*}
Q_1(x,y) &=& x^2 + b_0xy + c_0y^2\\
Q_2(x,y) &=& a_1x^2 + b_1xy + c_1y^2\\
Q_3(x,y) &=& a_1x^2 - b_1xy + c_1y^2\\
\vdots\\
Q_{2n}(x,y) &=& a_nx^2 + b_nxy + c_ny^2\\
Q_{2n+1}(x,y) &=& a_nx^2 - b_nxy + c_ny^2.
\end{eqnarray*}
Note that $Q_1$ is the principal form and $Q_{2i+1} = Q_{2i}^{-1}$ for $1 \leq i \leq n$. We know that we have this group structure because $h(\Delta)$ is odd, and so in particular there are no elements of order 2 in $C(\Delta)$.\\

Fix a form $Q_j$, where $j \not = 1$.  By Lemma \ref{UniqueSq}, there is a unique $k$ such that $Q_k \circ Q_k = Q_j$. Consider the composition
\[Q_2 \circ Q_3 \circ \dots \circ Q_{k-1} \circ Q_k \circ Q_k \circ Q_{k+2} \circ \dots \circ Q_{2n} \circ Q_{2n+1}\]
if $k$ is even, and the composition
\[Q_2 \circ Q_3 \circ \dots \circ Q_{k-2} \circ Q_k \circ Q_k \circ Q_{k+1} \circ \dots \circ Q_{2n} \circ Q_{2n+1}\]
if $k$ is odd. In either case, the composition reduces to 
\[Q_1 \circ \dots \circ Q_1 \circ Q_k \circ Q_k \circ Q_1 \circ \dots \circ Q_1 = Q_k \circ Q_k = Q_j.\]
As each form represents its own $x^2$ and $y^2$ coefficients, and each form has the same $x^2$ and $y^2$ coefficients as its inverse, by the composition law $\prod_{i=1}^{2n+1} a_ic_i$ is represented by $Q_j$ for $1 < j \leq 2n+1$.
Similarly, the composition
\[Q_2 \circ Q_3 \circ \ldots \circ Q_{2n+1} = Q_1 \]
shows that $\prod_{i=2}^{2n+1} a_ic_i$ is represented by $Q_1$. Also, if an integer $n$ is represented by a form $Q$ then $m^2n$ is represented by $Q$ for all $m \in \mathbb{Z}$. Thus the infinite set $\left\{m^2\left(\prod_{i=2}^{2n+1} a_ic_i\right): m \in \mathbb{Z} \right\} \subseteq \textrm{int}(S_{\Delta})$. 
\end{proof}

\textbf{Corollary \ref{OddClSizeCor}.} \textit{Let $Q_1$ and $Q_2$ be forms of discriminants $\Delta_1$ and $\Delta_2$ (respectively), and suppose that $h(\Delta_1)$ and $h(\Delta_2)$ are both odd. Then there are infinitely many $m$ represented by both $Q_1$ and $Q_2$.}

\begin{proof}
Let $S_{\Delta_1}$ be the set of forms of discriminant $\Delta_1$, $S_{\Delta_2}$ be the set of forms of discriminant $\Delta_2$, and $S_{\Delta_1\Delta_2}$ be the set of forms of discriminant $\Delta_1$ or $\Delta_2$. By Theorem \ref{OddClSize2} there exist integers $\alpha, \beta \not = 0$ such that $\alpha \in \textrm{int}(S_{\Delta_1})$ and $\beta \in \textrm{int}(S_{\Delta_2})$. Using composition with the identity element, $\alpha^2 \in \textrm{int}(S_{\Delta_1})$ and $\beta^2 \in \textrm{int}(S_{\Delta_2})$. 
We also know that if a form $Q$ represents an integer $n$, then it also represents $m^2n$ for all $m \in \mathbb{Z}$. Thus $\alpha^2\beta^2 \in \textrm{int}(S_{\Delta_1\Delta_2})$. Using this fact once again, we see that $\textrm{int}(S_{\Delta_1\Delta_2})$ is infinite.
\end{proof}

We now present two theorems concerning the local representation of integers. As mentioned earlier, if an integer $m$ is locally represented by a quadratic form $Q(x,y)=ax^2+bxy+cy^2$, then by the Hasse-Minkowski Theorem, $Q$ represents $m$ over $\mathbb Q$. By clearing denominators, we can easily show that $mn$ is represented by $Q$, where $n$ is a perfect square. 
The motivation behind the following theorems is to show that $Q$ represents a multiple $mn$, where $n$ is not a perfect square. 

\textbf{Theorem \ref{LocalOneD}.} \textit{Let $Q$ be a quadratic form, and suppose $m$ is locally represented by $Q$. Then some nonzero, nonsquare multiple of $m$ is represented by $Q$.}
\begin{proof}
By \cite[Theorem 1.3]{Cassels}, since $m$ is locally represented by $Q$, it is integrally represented by some form in the same genus as $Q$. Let $Q_i$ be the form that represents $m$. There exists a $Q_j$ such that $Q_i\circ Q_j = Q$, so 
$mn$ is represented by $Q$ for all $n$ represented by $Q_j$. Furthermore, since there is no binary quadratic form that represents only squares, we are guaranteed to find some multiple $mn$ where $n$ is not square.
\end{proof}

\textbf{Theorem \ref{LocalTwoD}:} \textit{Let $Q_1$ and $Q_2$ be forms with discriminants $\Delta_1$ and $\Delta_2$, respectively. Let $m$ be locally represented by $Q_1$ and $Q_2$, and let $n \not \equiv 0 \pmod{m}$ be represented by all forms of discriminant $\Delta_1$ and $\Delta_2$. Then $mk$ is represented by $Q_1$ and $Q_2$, where $k$ is a non-square integer.}
\begin{proof}
We know by Lemma \ref{Cox Lemma 2.5} that an odd prime $p$ not dividing $\Delta$ is represented by some form of discriminant $\Delta$ if and only if $\left(\frac{\Delta}{p}\right)=1$. Thus, an odd prime $p$ will be represented by some form of discriminant $\Delta_1$ and some form of discriminant $\Delta_2$ if and only if $\left(\frac{\Delta_1}{p}\right) = \left(\frac{\Delta_2}{p}\right)=1$. 
Let $p$ be an odd prime that satisfies $\left(\frac{\Delta_1}{p}\right) = \left(\frac{\Delta_2}{p}\right)=1$ and $\gcd(m,p)=1$. Let $n\not = 0$ be in the intersection of all the forms of discriminant $\Delta_1$ and $\Delta_2$. Since $n$ is represented by all forms of discriminant $\Delta_1$ and $p$ is represented by some form of discriminant $\Delta_1$, $np$ will be represented by all forms of discriminant $\Delta_1$ by Theorem \ref{POnetoAll}. Similarly, $np$ will be represented by all forms of discriminant $\Delta_2$. 

Now, let $m$ be locally represented by $Q_1$ and $Q_2$. Then we know by \cite[Theorem 1.3]{Cassels} that some form in the same genus as $Q_1$ represents $m$, and some form in the same genus as $Q_2$ represents $m$. Let $Q_{r}$ be the form of discriminant $\Delta_1$ that represents $m$, and let $Q_{s}$ be the form of discriminant $\Delta_2$ that represents $m$.  There will be some form, $Q_{g}$, such that $Q_{r}\circ Q_{g}=Q_1$, and there will be some form, $Q_{h}$, such that $Q_{s}\circ Q_{h}=Q_2$. We have already shown that $np$ is represented by all forms of discriminant $\Delta_1$ and $\Delta_2$, so $np$ will be represented by $Q_{g}$ and $Q_{h}$. By composing these forms with the principal form of their respective discriminants, we can show that $n^2p$ is represented by $Q_{g}$ and $Q_{h}$. Let $k=n^2p$, which is not a perfect square and is not a multiple of $m$. Thus by the composition law, $mk$ will be represented by both $Q_1$ and $Q_2$.
\end{proof}

\section{Class Field Theory}
\label{CFT}
	We now discuss specific material needed for the proof of Theorem \ref{number of ideals}. General references for this section include \cite[Chapters 5,6,7]{Cox} and \cite{Neukirch}. Let $K$ be a number field and $L$ a finite extension of $K$. Let $\mathcal{O}_K$ and $\mathcal{O}_L$  respectively denote the ring of integers of $K$ and $L$. If $\mathfrak{p}$ is a prime ideal in $\mathcal{O}_K$, then $\mathfrak{p}\mathcal{O}_L$ is an ideal of $\mathcal{O}_L$ with $\mathfrak{p}\mathcal{O}_L=\mathfrak{P}_1^{e_1}\cdots \mathfrak{P}_g^{e_g}$, for $\mathfrak{P}_i \subset L$ distinct prime ideals containing $\mathfrak{p}$. The integer $e_i$ is called the \textbf{ramification index} of $\mathfrak{p}$ in $\mathfrak{P}_i$. If $e_i=1$ for all $i$, we say that $\mathfrak{p}$ is \textbf{unramified} in $L$. Otherwise, $\mathfrak{p}$ is \textbf{ramified} in $L$. Each prime $\mathfrak{P}_i$ creates a finite residue field extension $\left[ (\mathcal O_L/ \mathfrak{P}_i) : (\mathcal O_K/ \mathfrak{p})\right]=f_i$. The degree of this extension is called the \textbf{inertial degree} of $\mathfrak{p}$ in $\mathfrak{P}$. If $f_i=1$ for all $i$, we say that $\mathfrak{p}$ \textbf{splits completely} in $L$.

	The extension of $K\subseteq L$ is \textbf{unramified} if all primes (including $p \vert \infty$) are unramified in $L$. The extension is \textbf{abelian} if $\text{Gal}(L/F)$ is an abelian group. The \textbf{Hilbert class field} $L$ of $K$ is the maximal, unramified, abelian extension of $K$. \\
	
	Henceforth, $K$ will refer to an imaginary quadratic number field, and $L$ to the Hilbert class field of $K$.\\
	
	Let $\mathfrak{p}$ be a prime of $\mathcal{O}_K$ and let $\mathfrak{P}$ be a prime of $\mathcal{O}_L$ containing $\mathfrak{p}$. Then the \textbf{Artin symbol} $((L/K)/ \mathfrak{P})$ is the unique element $\sigma\in \text{Gal}(L/K)$ such that for all $\alpha\in\mathcal{O}_L$, \[\sigma(\alpha)\equiv \alpha^{N(\mathfrak{p})}\pmod{\mathfrak{P}},\] where $N(\mathfrak{p})=|\mathcal{O}_K/\mathfrak{p}|$.
	
Last, given a form $Q$ of discriminant $\Delta=s^2n$ where $n<0$ is square-free, we can associate $Q$ to a field $K=\Q(\sqrt{n})$ of discriminant 
\[d_K=\begin{cases}
    
    n & \text{if } n\equiv 1\pmod{4},\\
    
    4n & \text{otherwise}.
    
    \end{cases} \]
    
	By Lemma \ref{Cox Lemma 2.5}, for $p\nmid\Delta$, $p$ is represented by some form of discriminant $\Delta$ if and only if $\legendre{\Delta}{p}=1$. Similarly, for $p\nmid d_k$, $p$ splits in $\mathcal{O}_K$ if and only if $\legendre{d_k}{p}=1$. Due to the relationship between $\Delta$ and $d_K$, $p\nmid \Delta$ is represented by some form of discriminant $\Delta$ if and only if $p$ splits in $\mathcal{O}_K$. Thus, there is a powerful relationship between splitting in fields and representation by forms. \\

Our final goal is to prove:\\

\textbf{Theorem \ref{number of ideals}.}
Let $\Delta=s^2n$, where $n<0$ is square-free, and let $Q$ be some form of discriminant $\Delta$ and order $m$. Let $K=\Q(\sqrt{n})$ and $L$ be the Hilbert class field of $K$. A prime $p$ not dividing the discriminant of $K$ is represented by some such $Q$ if and only if $p$ splits into $\frac{[L:\Q]}{m}$ factors in $L$.

Before we can provide the proof, we need a series of lemmas. 
\begin{restatable}{lemma}{Cox Lemma 5.9}
\label{Cox Lemma 5.9}
Let $\mathfrak{p}$ be prime in $K$. The primes $\mathfrak{P}_1,\ldots,\mathfrak{P}_g$ of $L$ containing $\mathfrak{p}$ all have  ramification index $e=1$ and inertial degree $f$. In addition, \[fg=[L:K].\]
\end{restatable}
\begin{proof}
Proof in \cite[Chapter 5, Section A]{Cox}.
\end{proof}
\begin{restatable}{lemma}{Cox Lemma 7.7}
\label{Cox Lemma 7.7}
Let $Q(x,y)=ax^2+bxy+cy^2$ be a form of discriminant $\Delta$. The map sending $Q$ to $[a,(-b+\sqrt{\Delta})/2]$ induces an isomorphism between the class group $C(\Delta)$ and the ideal class group $C(\mathcal{O}_K)$.
\end{restatable}
\begin{proof}
Proof in \cite[Chapter 5, Section D]{Cox}.
\end{proof}



\begin{restatable}{lemma}{Cox Lemma 5.23}
\label{Cox Lemma 5.23}
The Artin map \[\left(\frac{L/K}{\cdot}\right):I_K\to \text{Gal}(L/K),\] where $I_K$ is the set of ideals of $\mathcal{O}_K$, induces an isomorphism between $C(\mathcal{O}_K)$ and $\text{Gal}(L/K)$.
\end{restatable}
\begin{proof}
Proof in \cite[Chapter 5, Section C]{Cox}.
\end{proof}

\begin{restatable}{lemma}{Cox Lemma 5.21}
\label{Cox Lemma 5.21}
Let $\mathfrak{p}$ be a prime in $K$. Given a prime $\mathfrak{P}$ of $L$ containing $\mathfrak{p}$, the order of $\left(\frac{L/K}{\mathfrak{P}}\right)=f$, the inertial degree of $\mathfrak{p}$.
\end{restatable}
\begin{proof}
Proof in \cite[Chapter 5, Section C]{Cox}.
\end{proof}

We can now prove Theorem \ref{number of ideals}.
\begin{proof}
Let $\Delta=s^2n$, where $n<0$ is square-free. Let $K=\Q(\sqrt{n})$ and $L$ be the Hilbert class field of $K$.

	To prove the forward direction, assume that $p$ is a prime represented by a form $Q$ of discriminant $\Delta$ and order $m$ in the class group. Then $p$ splits in $\mathcal{O}_K$. Let $p\mathcal{O}_K=\mathfrak{p}_1\mathfrak{p}_2$ be the prime factorization of $p$ in $\mathcal{O}_K$. Suppose that $\mathfrak{p}_1$ splits into $\mathfrak{P}_{1,1}^{e_{1,1}}\ldots \mathfrak{P}_{1,g}^{e_{1,g}}$ and $\mathfrak{p}_2$ splits into $\mathfrak{P}_{2,1}^{e_{2,1}}\ldots \mathfrak{P}_{2,h}^{e_{2,g}}$ in $L$. We know that they split into the same number of primes because $L$ is a Galois extension of $\mathbb{Q}$.

	We know that since $L$ is an unramified extension of $K$, $e_{i,j}=1$ for all $i,j$. Since $L$ is a Galois extension of $K$, we also know that $f_{1,i}=f_{1,j}$ and $f_{2,i}=f_{2,j}$ for all $i,j$ by Lemma \ref{Cox Lemma 5.9}. Let $f_{1,i}=f_1$ and $f_{2,i}=f_2$ for all $i$. By Lemma \ref{Cox Lemma 7.7}, there is an isomorphism between the form $Q$ that represents $p$ and the ideal $[a,(-b+\sqrt{\Delta})/2]$. We can see that this ideal is equal to $\mathfrak{p}_1$.

	By Lemma \ref{Cox Lemma 5.23}, there is also an isomorphism between the ideal $\mathfrak{p}_1$ and $\left(\frac{L/K}{\mathfrak{p}_1}\right)$. Hence the order of the Artin symbol, denoted $\left|\left(\frac{L/K}{\mathfrak{p}_1}\right)\right|$, is the same as the order of $Q$ in the class group. The Artin symbol has order $f_{1}$ by Lemma \ref{Cox Lemma 5.21}. We also know that $[L:K]=f_1g$ by Lemma  \ref{Cox Lemma 5.9}. Thus,

\[ g=\frac{[L:K]}{f_1}=\frac{[L:K]}{\left|\left(\frac{L/K}{\mathfrak{p}_1}\right)\right|}=\frac{[L:K]}{m}.\]

Hence the prime $p$ splits into $n$ factors, where
\[n=g+h=\frac{2[L:K]}{m}=\frac{[L:\Q]}{m}.\]

To prove the backwards direction, assume that $p$ splits into $n$ factors in $L$. Since all of the homomorphisms in the forward direction were isomorphisms, this direction follows similarly.
\end{proof}

\section{Examples}
\label{Examples}
\textbf{Example of Corollary \ref{OddClSizeCor}}:
Let $\Delta_1 = -47$ and $\Delta_2 = -23$. The forms of $S_{\Delta_1}$ are
\begin{eqnarray*}
Q_1(x,y) = x^2 + xy + 12y^2\\
Q_2(x,y) = 2x^2 + xy + 6y^2\\
Q_3(x,y) = 2x^2 - xy +6y^2\\
Q_4(x,y) = 3x^2 + xy +4y^2\\
Q_5(x,y) = 3x^2 - xy +4y^2.
\end{eqnarray*}
The class group $C(\Delta_1) \cong (\mathbb Z / 5 \mathbb Z)$, and we have:
\begin{eqnarray*}
Q_1 &=& Q_2\circ Q_3\circ Q_4\circ Q_5\\
Q_2 &=& Q_2\circ Q_3\circ Q_4\circ Q_4\\
Q_3 &=& Q_2\circ Q_3\circ Q_5\circ Q_5\\
Q_4 &=& Q_4\circ Q_5\circ Q_3\circ Q_3\\
Q_5 &=& Q_4\circ Q_5\circ Q_2\circ Q_2.
\end{eqnarray*}
As $2$ and $6$ are the $x^2$ and $y^2$ coefficients, respectively, of $Q_2$ and $Q_3$, they are represented by $Q_2$ and $Q_3$. Similarly $3$ and $4$ must be represented by $Q_4$ and $Q_5$. 
The above compositions imply that $2\cdot 6\cdot 3\cdot 4 = 144$ is represented by $Q_1,\ldots, Q_5$. Also, we can compose each form with the principal form, $Q_1$, to show that $144^2 \in \textrm{int}(S_{\Delta_1})$.

The elements of $S_{\Delta_2}$ are:
\begin{eqnarray*}
Q_6 &=& x^2 + xy + 6y^2\\
Q_7 &=& 2x^2 + xy + 3y^2\\
Q_8 &=& 2x^2 - xy +3y^2.
\end{eqnarray*}
Here $C(\Delta_2) \cong (\mathbb Z / 3 \mathbb Z)$ and one can see
\begin{eqnarray*}
Q_6 &=& Q_7\circ Q_8\\
Q_7 &=& Q_8\circ Q_8\\
Q_8 &=& Q_7\circ Q_7.
\end{eqnarray*}
Now, since 2 and 3 are the $x^2$ and $y^2$ coefficients, respectively, of $Q_7$ and $Q_8$, they are represented by $Q_7$ and $Q_8$. Thus 
$6$ is represented by $Q_6,Q_7,$ and $Q_8$. Also, we can compose each form with the principal form to show that $6^2 \in \textrm{int}(S_{\Delta_2})$.\\

Now, if a form $Q$ represents an integer $m$, then it also represents $n^2m$ for all $n\in \mathbb{Z}$.  Thus $6^2\cdot 144^2 \in \textrm{int}(S_{\Delta_1} \cup S_{\Delta_2})$.

\textbf{Example of Theorem \ref{number of ideals}:} This result was developed as a tool for answering a more general question: Let $p \neq q$ be distinct primes, and consider all forms of the type $px^2+bxy+qy^2$. What primes beyond $p$ and $q$ are represented by all forms?\\

In our example we let $p=2$ and $q=11$. Then we have the following four forms: 

\begin{center}
\begin{tabular}{c||c|c}
 Form & $\Delta(Q)$ & $C(\Delta)$  \\ \hline
$Q_1(x,y) =2x^2+11y^2$ & -88 & $\Z/2\Z$\\
$Q_2(x,y)=2x^2+xy+11y^2$ & -87 & $\Z/6\Z$\\
$Q_3(x,y)=2x^2-xy+11y^2$ & -87 & $\Z/6\Z$\\
$Q_4(x,y)=2x^2+2xy+11y^2$ & -84 & $(\Z/2\Z)^2$\\
\end{tabular}\\
\end{center}

Due to Clark et al. \cite{GoNI}, we can classify which primes $p \nmid 2\Delta$ are represented by $Q_1$ and $Q_4$ using modular conditions, and we find that a prime $p\neq 3,7$ will be represented by both $Q_1$ and $Q_4$ if and only if $\legendre{p}{11}=-1$, $\legendre{p}{7}=1$, and $p\equiv 11 \pmod{24}$.\\

While modular conditions are not sufficient to determine the primes represented by $Q_2$ and $Q_3$ they are improperly equivalent and hence represent exactly the same integers. The imaginary quadratic number field associated to them is $K=\Q(\sqrt{-87})$. The Hilbert class field $L$ is a degree 6 extension of $K$ and a degree 12 extension of $\Q$. As $Q_2$ has order 6 in the class group, by Theorem \ref{number of ideals} a prime $p$ not dividing the discriminant of $K$ represented by $Q_2$ (or $Q_3$) splits into 2 factors in $L$.\\

Since $Q_2$ and $Q_3$ are the only forms of order $6$ in $C(-87)$, a prime $p\nmid-87$ is represented by $Q_2$ and $Q_3$ if and only if $p$ splits into two factors in $L$. So now we can say that a prime $p \neq 2,3,7,11,29$ is represented by all of the forms $2x^2+bxy+11y^2$ if and only if $\legendre{p}{11}=-1$,$\legendre{p}{7}=1$, $p\equiv 3\pmod{8}$, $p\equiv 2 \pmod {3}$, and $p$ splits into two factors in the Hilbert class field of $K=\Q(\sqrt{-87})$. We are guaranteed an infinite number of these primes by the Chebotarev Density Theorem. The smallest prime for which all of these conditions hold is $659$.


\end{document}